\documentclass[twoside, a4paper, 11pt]{article}
\usepackage{cite, amsmath, amssymb, booktabs, lastpage, graphicx}
\usepackage[hidelinks]{hyperref}
\usepackage{geometry}
\usepackage{tikz}
\usepackage{enumitem}
\usepackage{pgfplots}
\usepackage{pgfkeys}
\pgfplotsset{compat= newest, width=10cm}
\usepackage{setspace}
\usepackage{amsthm}
\theoremstyle{plain}
\newtheorem{lemma}{Lemma}[section]
\newtheorem{theorem}{Theorem}[section]

\newtheorem{proposition}{Proposition}[section]
\newtheorem*{corollary}{Corollary}

\theoremstyle{remark}

\theoremstyle{definition}
\newtheorem{definition}{Definition}[section]
\newtheorem{example}{Example}[section]

\usepackage{graphicx}
\makeatletter
\renewcommand{\maketitle}{
	\begin{center}
		\rule[1em]{\textwidth}{0mm}
		\baselineskip=0.30in
		{\bfseries \@title} \par
		\vspace{5mm}
		\baselineskip=0.2in
		{\bfseries \@author}\par
		\vspace{1mm}
		{\it \@address} \par
		{\small\tt \@email} \par
		\vspace{3mm}
		
	\end{center}
	\vspace{3mm}
}
\newcommand{\address}[1]{\def\@address{#1}}
\newcommand{\email}[1]{\def\@email{#1}}

\usepackage{fancyhdr}
\title{On the Clean Graph of a Ring}
\author{Randhir Singh$^{a,*}$, S. C. Patekar$^b$}
\address{$^a$Center for Advanced Studies in Mathematics, Department of Mathematics, Savitribai Phule Pune University, Pune-411007, India\\
	$^b$Center for Advanced Studies in Mathematics, Department of Mathematics, Savitribai Phule Pune University, Pune-411007, India\\
}
\email{$^a$randhir\textunderscore singh46@yahoo.com, $^b$shri82patekar@gmail.com}
\date{\today}
\begin{document}
	\maketitle
	\vskip 2mm	
	\begin{abstract}
		Let $R$ be a ring (not necessarily commutative ring) with identity. The clean graph $Cl(R)$ of a ring $R$ is a graph with vertices in the form of ordered pair $(e,u)$, where $e$ is an idempotent of the ring $R$ and $u$ is a unit of the ring $R$. Two distinct vertices $(e,u)$ and $(f,v)$ are adjacent if and only if $ef=fe=0$ or $uv=vu=1$. In this paper, we determine the Wiener index, Matching number of the clean graph of the ring $\mathbb{Z}_n$.
		
		\vskip 1mm

	\end{abstract}
	\section{Introduction}
	Studying the graphs associated with the algebraic structure to understand various properties is an interesting and classical technique. In this way, various graph parameters can be studied using algebraic properties for better understanding. For example, one can study graphs associated with ring $R$ with unity.
	
	The concept of the zero divisor graph of a commutative ring was introduced by Beck\cite{IB}. He studied the problem of coloring of the graph of a commutative ring with unity. Anderson et al. \cite{DP} modified the definition of the zero divisor graph of a commutative ring with unity as follows: let $R$ be a commutative ring with unity and zero divisor graph $\Gamma (R)$ be a simple undirected graph. Vertex set $V(\Gamma(R))$ consist of non-zero zero divisors of the commutative ring $R $ with unity and two distinct vertices $x$ and $y$ are adjacent if and only if $xy=0$. Many researchers are attracted to study the graphs associated with algebraic structure.
	
	\section{Preliminary}
	Following are some notions and definitions used throughout the paper. \\If $R$ is any ring and $ A\subset R$ is any subset of $R$, then $A^*=A\backslash\{0\}$. For the ring $R$ (not necessarily commutative), $Id(R)$ denotes the set of idempotents, $U(R)$ denotes the set of units. Set of units in $R$ which are self invertible is denoted by $U'(R)$. Set of units in $R$ which are not self invertible is $U''(R)=U(R)\backslash U'(R)$. For a graph $ G $, $V(G)$ denotes the vertex set of graph $G$ and $E(G)$ denotes the edge set of graph $G$. If vertices $u$ and $v$ are adjacent, we denote it by $u \sim v$.
	
	
	\begin{definition}
		A matching in a graph $G$ is a set of non-loop edges with no shared endpoints. The vertices incident to the edges of matching $M$ are saturated by $M$. A perfect matching in a graph is a matching that saturates every vertex. A matching number, denoted by $\mu(G),$ is the maximum size of a matching in $G.$ Graph contains a perfect matching, if $\mu(G)=\frac{|V|}{2}.$
	\end{definition}
	
	\begin{definition} 
		\textbf{\cite{WK}} An element of a ring is said to be clean if it can be written as the sum of a unit and an idempotent. If all the elements of the ring are clean then that ring is said to be a clean ring.
	\end{definition}  
	
	Clean graph $Cl(R)$ of a commutative ring $R$ is defined by Habibi et al.\cite{Habibi} as follows:
	
	\begin{definition}\textbf{\cite{Habibi}}\label{Clean Graph}
		The clean graph of a ring $R$ denoted by $Cl(R)$ is undirected graph with vertex set $V(Cl(R))=\{(e,u):e\in Id(R) \text{ and } u\in U(R) \}$. Two distinct vertices $(e,u)$ and $(f,v)$ are adjacent if and only if $ef=fe=0$ or $uv=vu=1$. The subgraph induced by set $\{(e,u): e(\neq0)\in Id(R) \text{ and } u\in U(R)\}$ is denoted by $Cl_2(R)$. 
	\end{definition}

	\begin{definition}
		Let $ x,y \in V(G) $. The distance between $x$ and $y$ in $ G $, denoted by $d(x,y)$ is the number of edges on the shortest path between $x$ and $y$. If there is no path between $ x $ and $ y $ then $ d(x,y)=\infty $ .
	\end{definition}
	
	\begin{definition}
		The Wiener index is a distance-based topological index introduced by H. Wiener\cite{HW} and defined to be sum of all distances between all pairs of vertices of graph $G$. Hosoya\cite{Hosoya} gave the mathematical representation for the Wiener index of graph $G$, defined by
		
		\begin{center}
			$W(G)=\frac{1}{2}\sum\limits_{v\in V(G)}d(u|G)$.
		\end{center}
		where, $d(u|G)$ denotes the distance of a vertex $u$ in graph $G$. 
	\end{definition}
	
	
	\begin{definition}
		Let $R$ be a ring with unity and an element $a\in R$ is said to be an idempotent if $a^2=a$. $ 0 $ and $ 1 $ are called as trivial idempotents. Moreover, if $a$ is an idempotent then $1-a$ is also an idempotent.
	\end{definition}
	
	\begin{definition}
		Euler's phi (or totient) function of a positive integer $ n $, denoted by $\phi(n)$ is the set of all positive integers less than $ n $ which are relatively prime to $ n $. 
		
	\end{definition}
	
	If $n$ is a positive integer with prime factorization, $n = \prod_{i=1}^{k}{p_i}^{\alpha_i}$, then $\phi(n)=\prod_{i=1}^{k}({p_i}^{\alpha_i}-{p_i}^{\alpha_i-1})$.
	Hence the set of units in $\mathbb{Z}_n$ will be $U(\mathbb{Z}_n)=\{u_1=1,u_2,u_3,...u_r,u_{r+1},u_{r+2},...,u_{\phi(n)-1},u_{\phi(n)}\}$. 
	
	\begin{proposition}\textbf{\cite{Hewitt}} \label{Prop 1}
		There are $2^k$ idempotents in $\mathbb{Z}_n$ and $2^k-1$ non-zero idempotents are there in $\mathbb{Z}_n$, where $k$ is the number of distinct primes dividing $n$.

	\end{proposition}
	
	$Id({\mathbb{Z}_n})^*=\{ \mathfrak{e}_1,\mathfrak{e}_2,\mathfrak{e}_3,,...,\mathfrak{e}_{2^k-2},\mathfrak{e}_{2^k-1} \},$ where $\mathfrak{e}_1= 1$ and $\mathfrak{e}_{2m+1} = 1 - \mathfrak{e}_{2m}$ for $m \in \{ 1,2, ...,\frac{2^{k}-4}{2} , \frac{2^{k}-2}{2} \}$
	
	\begin{theorem}\label{T1}\cite{Habibi}
		$Cl_2(R)$ is connected if and only if $R$ has a non trivial idempotents. Moreover, if $Cl_2(R)$ is connected then $diam(Cl_2(R))\leq 3$.
	\end{theorem}
	
	In this paper, we are considering $Cl_2(\mathbb{Z}_n)$.
	
	\section{Main Result}
	\begin{proposition}\label{Prop 3} For any non negative integer $m$, $n=2^m\prod\limits_{i=1}^{k}{p_i}^{\alpha_i}$, $p_i\neq2$ and $p_i$'s are distinct primes, in ring $\mathbb{Z}_n$, \begin{center}
			$|U'(\mathbb{Z}_n)|=\begin{cases}
				2^k & \text{ if }m=0 \text{ or }m=1 \\
				2^{k+1} & \text{ if } m=2\\
				2^{k+2} & \text{ if } m\geq 3.
			\end{cases}$
		\end{center}
		Note that, for our convenience we write $|U'(\mathbb{Z}_n)|=r$. The value of $r$ changes accordingly.
		
	\end{proposition}
	\begin{proof}
		To find number of self invertible element $|U'(\mathbb{Z}_n)|$ in $ \mathbb{Z}_n $, we find the solution for $x^2\equiv1(mod ~n)$, where $n=2^m\prod\limits_{i=1}^{k}{p_i}^{\alpha_i}$, $p_i\neq2.$\\
		$x^2\equiv1(mod ~n)$ implies
		$x^2\equiv1(mod ~2^m\prod\limits_{i=1}^{k}{p_i}^{\alpha_i}).$\\
		Therefore, 
		
		\begin{equation*}\label{eq1}
			x^2\equiv1(mod~2^m) \text{ and }  x^2\equiv1(mod \prod\limits_{i=1}^{k}{p_i}^{\alpha_i})                
		\end{equation*}
		The solutions of $x^2\equiv1(mod~2^m)$ are as follows:
		\begin{enumerate}
			\item  $x\equiv1(mod~2)$, for $m=1$.
			\item $x\equiv1(mod~4)$ and $x\equiv3(mod~4)$, for $m=2$.
			
			\item 	The solution set of equation $x^2\equiv1(mod~2^m)$, for $m\geq3$, is multiplicative group $U_{2^m}$ of units of $\mathbb{Z}_{2^m}$. By Theorem $20$ \cite{Jacobson}, $U_{2^m}$ is direct product of a cyclic group of order $2$ and one of order $2^{m-1}$. $U_{2^m}\cong \mathbb{Z}_{2}\times\mathbb{Z}_{2^{m-1}}$ and there are $4$ elements of order $2$.
			
		\end{enumerate}
		The solutions of $x^2\equiv1(mod ~{p_i}^{\alpha_i})$ are $x\equiv1(mod ~{p_i}^{\alpha_i})$ and $x\equiv p-1(mod ~{p_i}^{\alpha_i})$.
		Therefore, by Chinese remainder theorem,
		
		\begin{center}
			$|U'(\mathbb{Z}_n)|=\begin{cases}
				2^k & \text{ if }m=0 \text{ or }m=1 \\
				2^{k+1} & m=2\\
				2^{k+2} & m\geq 3.
			\end{cases}$
		\end{center}
		
	\end{proof}
	\begin{proposition}\label{Prop 2}
		
		Let $n=\prod_{i=1}^{k}{p_i}^{\alpha_i}$, $\alpha_i \in \mathbb{N}$, for $i=1,2,...,k$. The following statements hold true:
		\begin{enumerate}
			\item $diam(Cl_2(\mathbb{Z}_n))=\infty$ if $k=1$.
			\item $diam(Cl_2(\mathbb{Z}_n))=3$ if $k\geq2$.
		\end{enumerate}
		
	\end{proposition}
	\begin{proof}
		Let $n=\prod\limits_{i=1}^{k}{p_i}^{\alpha_i}$, $\alpha_i\in \mathbb{N}$ for $i=1,2,...,k$ and $\mathbb{Z}_n=\{\overline{0},\overline{1},\overline{2},...,\overline{n-1}\}.$ Let $v_i=(e,u)$ and $v_j=(f,v) $ be the distinct vertices of $ Cl_2(\mathbb{Z}_n) $. To get $diam(Cl_2(\mathbb{Z}_n))$, we make the following cases:
		\paragraph{Case 1.}
		$k=1$: In this case, $n={p_1}^{\alpha_1}$ and $ \mathbb{Z}_{{p_1}^{\alpha_1}} $ is local ring. It is clear that there are only trivial idempotents $ 0 $ and $ 1 $ in $ \mathbb{Z}_{{p_1}^{\alpha_1}} $. We observe that $ (1,1)$ is an isolated vertex of $ Cl_2(\mathbb{Z}_{{p_1}^{\alpha_1}}) .$ Therefore, $diam(Cl_2(\mathbb{Z}_n))=\infty$.
		
		\paragraph{Case 2.}
		$k\geq 2$: In this case, $n=\prod\limits_{i=1}^{k}{p_i}^{\alpha_i}$ and $ \mathbb{Z}_n=\{\overline{0}, \overline{1}, \overline{2},...,\overline{n-1}\}.$  So, $V(G)=\{(e,u):e\in Id(\mathbb{Z}_n)^*, u \in U(\mathbb{Z}_n)\}.$ By Theorem \ref{T1}, $diam(Cl_2(R))\leq 3.$
		
		Claim: $diam(Cl_2(\mathbb{Z}_n))=3.$
		
		Consider the vertices $(1,u)$ and $(1,v)$, $uv\neq1$, then by Definition\ref{Clean Graph}, $(1,u)\sim (1-g,u^{-1})\sim (g,v^{-1})\sim(1,v).$ Thus, $diam(Cl_2(\mathbb{Z}_n))=3.$

	\end{proof}

	\begin{lemma}\label{Lemma 1}
		If $Cl_2(\mathbb{Z}_n)$ is connected graph
		then for $v_i=(e,u),v_j=(f,v) \in V(Cl_2(\mathbb{Z}_n))$, we have
		
		\begin{center}
			$d(v_i,v_j)=\begin{cases}
				1  &  \text{ if } ef=0  \text{ or }  uv=1 , u \neq v, e \neq f\\
				2  &  \text{ if }  e=f\neq 1 ,  uv\neq1 \\
				2  & \text{ if }   ef \neq 0  \text{ and }  uv \neq 1 ,  u\neq v ,  e\neq f \\	
				3  & \text{ if }  e=f=1 ,  u\neq v ,  uv\neq 1 .
			\end{cases}$
		\end{center} 
	\end{lemma}
	\begin{proof} Let $v_i=(e,u)$, $ v_j=(f,v) \in  V(Cl_2(\mathbb{Z}_n))$ be the distinct vertices, where $e,f\in Id(\mathbb{Z}_n)^*$ and $u,v \in U(\mathbb{Z}_n)$. To get the the distance $d(v_i,v_j)$ between $v_i$ and $v_j$, we make the following cases:
		\paragraph{Case 1}$ ef=0 $ or $ uv=1 $ , $ u \neq v$, $e \neq f$: By Definition \ref{Clean Graph}, vertices $v_i\sim v_j$ and hence $d(v_i,v_j)=1$. 

		\paragraph{Case 2} $e=f\neq  1$ and $uv\neq 1$: In this case, by Definition \ref{Clean Graph}, $(e,u)\nsim (e,v)$. Therefore, $d(v_i,v_j)\geq1.$ So,  $ (e,u)\sim(1-e,u)\sim(e,v) $ hence $d(v_i,v_j)=2$.  
		
		\paragraph{Case 3} $ef \neq 0$ and $uv \neq 1$, $u\neq v$, $e\neq f$: In this case, by Definition \ref{Clean Graph}, vertices $v_i \nsim v_j$. So, $1< d(v_i,v_j)\leq3.$ It is clear that, $(e,u)\sim(1-f,u^{-1})\sim(f,v)$ and $d(v_i,v_j)=2$

		\paragraph{Case 4} $e=f=1$, $u\neq v$, $uv\neq 1$: By Definition \ref{Clean Graph}, $(1,u)\nsim(1,v)$. It is clear that, $(1,u)\sim (1-g,u^{-1})\sim(g,v^{-1})\sim(1,v)$ and hence $d(v_i,v_j)=3$.   		
	\end{proof}
	
	\begin{theorem}\label{thm 1}
		Let $n=\prod_{i=1}^{k}{p_i}^{\alpha_i}$, $\alpha_i \in \mathbb{N}$ for $i=1,2,...,k$ and $p_i$'s be distinct primes. The Wiener index of $Cl_2(\mathbb{Z}_n)$ is defined by
		\begin{enumerate}
			\item $W(Cl_2(\mathbb{Z}_n))=\infty$ if $k=1$.
			\item $W(Cl_2(\mathbb{Z}_n))=\frac{1}{2}\left[\phi(n)^2\left(2(2^{2k})-5(2^k)+5\right)-\phi(n)\left(2^{2k}-2^k+3\right)+2^kr\right]$ if $k\geq2$.
		\end{enumerate}
	\end{theorem}
	\begin{proof} For $n=\prod_{i=1}^{k}{p_i}^{\alpha_i}$, $\alpha_i \in \mathbb{N}$, $\mathbb{Z}_n=\{\overline{0},\overline{1},\overline{2},...,\overline{n-1}\}$. To get the Wiener index of $Cl_2(\mathbb{Z}_n)$, we make the following cases: 
		\paragraph{Case 1.} $k=1$: $n={p_1}^{\alpha_1}$ and $\mathbb{Z}_{{p_1}^{\alpha_1}}=\{\overline {0}, \overline{1},...,\overline {{p_1}^{\alpha_1}-1}\}$ is local ring that contains only trivial idempotents $0$ and $1$. So, $Id(\mathbb{Z}_n)^*=\{1\}$ and $U(\mathbb{Z}_{{p_1}^{\alpha_1}})=\{u_1=1, u_2,...,u_{\phi({p_1}^{\alpha_1})}\}$. By Definition \ref{Clean Graph}, vertex $(1,1)\in Cl_2(\mathbb{Z}_{{p_1}^{\alpha_1}})$ is an isolated vertex, therefore $Cl_2(\mathbb{Z}_{{p_1}^{\alpha_1}})$ is disconnected. Thus, $W(Cl_2(\mathbb{Z}_{{p_1}^{\alpha_1}}))=\infty.$ 
		
		\paragraph{Case 2.}  $k\geq2$: $n=\prod_{i=1}^{k}{p_i}^{\alpha_i}$ and $\mathbb{Z}_n=\{\overline {0}, \overline{1},...,\overline {n-1}\}$. So, by Proposition \ref{Prop 1}, there are $2^k-1$ non-zero idempotents. Note that, $Id({\mathbb{Z}_n})^*=\{ \mathfrak{e}_1,\mathfrak{e}_2,\mathfrak{e}_3,,...,\mathfrak{e}_{2^k-2},\mathfrak{e}_{2^k-1} \},$ where $\mathfrak{e}_1= 1$ and $\mathfrak{e}_{2m+1} = 1 - \mathfrak{e}_{2m}$ for $m \in \{ 1,2, ...,\frac{2^{k}-4}{2} \frac{2^{k}-2}{2} \}$ and  $U(\mathbb{Z}_n)=\{u_1=1,u_2,u_3,...,u_r,u_{r+1},u_{r+2},..., u_{(\phi(n)-1)}, u_{\phi(n)}\}$. Now, we assume $U'(\mathbb{Z}_n)=\{u_1,u_2,...,u_r\}$. So, $\left|U'(\mathbb{Z}_n)\right|=r$. 
		and $\left|U''(\mathbb{Z}_n)\right|=\phi(n)-r$. Note that $\phi(n)-r$ is always even as for each $u\in U''(\mathbb{Z}_n)$ there exist unique $v\in U''(\mathbb{Z}_n)$ such that $uv=1$. 
		Therefore, there are $\frac{\phi(n)-r}{2}$ pairs of such $u$ and $v$ in $ U''(Z_n)$ such that $(e,u)\sim(f,v)$. 

		We may write $V(Cl_2(\mathbb{Z}_n))=V_1\sqcup V_2\sqcup...\sqcup V_{2^k-1}$, where $V_i=\{(\mathfrak{e}_i,u):\mathfrak{e}_i \in Id({\mathbb{Z}_n})^*, u \in U(\mathbb{Z}_n)\}$ and $|V_i|=\phi(n)$.\\ 
		Now, we find Wiener index of $Cl_2(\mathbb{Z}_n)$ 
		
		$W(Cl_2(\mathbb{Z}_n))=\frac{1}{2}\sum\limits_{u\in Cl_2(\mathbb{Z}_n)}d(u|Cl_2(\mathbb{Z}_n))$\\
		
		$=\sum\limits_{l=1}^{2^k-1}\sum\limits_{v_i,v_j\in V_l}d(v_i,v_j)+\sum\limits_{\substack{v_i\in V_p\\v_j\in V_q\\p\neq q}}d(v_i,v_j)$ \\
		
		$=\sum\limits_{v_i,v_j\in V_1} d(v_i,v_j)+\sum\limits_{l=2}^{2^k-1}\sum\limits_{v_i,v_j\in V_l}d(v_i,v_j)+ \sum\limits_{\substack{v_i\in V_1\\v_j\in V_q\\q \neq 1}}d(v_i,v_j)+\sum\limits_{\substack{v_i\in V_p\\v_j\in V_q\\p\neq q\neq 1}}d(v_i,v_j)$.\\
		
		Set $S_1=\sum\limits_{v_i,v_j\in V_1} d(v_i,v_j)$, $S_2=\sum\limits_{l=2}^{2^k-1}\sum\limits_{v_i,v_j\in V_l}d(v_i,v_j)$, $S_3=\sum\limits_{\substack{v_p\in V_1\\v_q\in V_q\\q \neq 1}}d(v_p,v_q)$, $S_4=\sum\limits_{\substack{v_p\in V_p\\v_q\in V_q\\p\neq q\neq 1}}d(v_p,v_q)$. 
		
		Now, we determine value of each summation as follows.\\
		\begin{enumerate}

			\item 	$ S_1=\sum\limits_{v_i,v_j\in V_1} d(v_i,v_j)$\\
			
			Since, $V_1=\{(1,u):1 \in Id(\mathbb{Z}_n)^*, u \in U(\mathbb{Z}_n)\}=\{v_1=(1,1),v_2=(1,u_2), v_3=(1,u_3),...,$\\$v_r=(1,u_r),v_{r+1}=(1,u_{r+1}),...,v_{\phi(n)}=(1,u_{\phi(n)})\}$.
			

				There are $\frac{\phi(n)-r}{2}$ pairs of vertices of type $v_i=(\mathfrak{e}_l ,u)$ and $v_j=(\mathfrak{e}_l ,v)$ in $V_l$ for all $l=1,2,3,...,$\\$2^k-2,2^k-1$ such that $v_i \sim v_j$ as $uv=1$. By Lemma \ref{Lemma 1}, if $uv=1$ and $e=f=1$,  $d(v_i,v_j)=1$. So that, for $l=1$,

				\begin{equation}\label{eq1}
					\sum\limits_{\substack{v_i,v_j\in V_1\\uv=1}}d(v_i,v_j)=\left[\frac{(\phi(n)-r)}{2}\right]
				\end{equation}
				
				Now assume there are no vertices $v_i=(1,u)$ and $v_j=(1,v)$ in $V_1$ such that $uv=1$. By Lemma \ref{Lemma 1}, if $uv\neq1$ and $e=f=1$, then $d(v_i,v_j)=3$. Therefore, by fixing $i$, $\sum d(v_i,v_j)=3(\phi(n)-i)$, for all  $j>i$ and $i=1,2,3,...,\phi(n)-1$. So that,

				\begin{equation}\label{eq2}
					\sum\limits_{\substack{v_i,v_j\in V_1\\uv\neq1}}d(v_i,v_j)=\sum\limits_{\substack{v_i,v_j\in V_1\\uv\neq1\\i=1}}^{\phi(n)-1}\sum\limits_{j>i} d(v_i,v_j)= 3\left[(\phi(n)-1)+(\phi(n)-2)+...+1\right].
				\end{equation}	
				
				Therefore, from $(\ref{eq1})$ and $(\ref{eq2})$ we get. \\
				
				$\sum\limits_{v_i,v_j\in V_1} d(v_i,v_j)=\sum\limits_{\substack{v_i,v_j\in V_1\\uv\neq1}}d(v_i,v_j)-2\sum\limits_{\substack{v_i,v_j\in V_1\\uv=1}}d(v_i,v_j)$ \\
				
				\hspace{1.04in}$=3\left[(\phi(n)-1)+(\phi(n)-2)+...+1\right]-2\left[\frac{(\phi(n)-r)}{2}\right]$\\

				After simplification, we get
				\begin{equation}\label{eq3}
					S_1=\sum\limits_{v_i,v_j\in V_1} d(v_i,v_j)=\frac{1}{2}\left[3\phi(n)^2-5\phi(n)+2r\right].
				\end{equation}

				
				\item 	$S_2=\sum\limits_{l=2}^{2^k-1}\sum\limits_{v_i,v_j\in V_l}d(v_i,v_j).$
				
				Since, $V_l=\{(\mathfrak{e}_l,u):\mathfrak{e}_l \in Id({\mathbb{Z}_n})^*, u \in U(\mathbb{Z}_n)\}=\{v_1=(\mathfrak{e}_l,1),v_2=(\mathfrak{e}_l,u_2), v_3=(\mathfrak{e}_l,u_3),...,v_r=(\mathfrak{e}_l,u_r),v_{r+1}=(\mathfrak{e}_l,u_{r+1}),...,v_{\phi(n)}=(\mathfrak{e}_l,u_{\phi(n)})\}$.
				
				Similarly, assume there are no vertices $v_i=(\mathfrak{e}_l,u)$ and $v_j=(\mathfrak{e}_l,v)$ in $V_l$ such that $uv=1$. By Lemma \ref{Lemma 1}, if $uv\neq1$ and $e=f\neq1$, then $d(v_i,v_j)=2$. Therefore, by fixing $i$, $\sum d(v_i,v_j)=2(\phi(n)-i)$, for all  $j>i$ and $i=1,2,3,...,\phi(n)-1$. So that,
				
				\begin{equation}\label{eq4}
					\sum\limits_{\substack{v_i,v_j\in V_l\\uv\neq1\\l\neq1}}d(v_i,v_j)	=\sum\limits_{\substack{v_i,v_j\in V_l\\uv\neq1\\i=1}}^{\phi(n)-1}\sum\limits_{j>i} d(v_i,v_j)= 2\{(\phi(n)-1)+(\phi(n)-2)+...+1\}
				\end{equation}
				
				Therefore, from $(\ref{eq1})$ and $(\ref{eq4})$ we get. \\
				
				$\sum\limits_{v_i,v_j\in V_l}d(v_i,v_j)= \sum\limits_{\substack{v_i,v_j\in V_l\\uv\neq1\\l\neq1}}d(v_i,v_j)-\sum\limits_{\substack{v_i,v_j\in V_l\\uv=1}}d(v_i,v_j)=2\{(\phi(n)-1)+(\phi(n)-2)+...+1\}
				-\left(\frac{\phi(n)-r}{2}\right) $\\
				
				So the distance between all pair vertices $v_i=(\mathfrak{e}_l,u)$ and $v_j=(\mathfrak{e}_l,v)$ in $V_l$ for all $l=2,3,...,2^k-2,2^k-1$ is
				
				$S_2=\sum\limits_{l=2}^{2^k-1}\sum\limits_{v_i,v_j\in V_l}d(v_i,v_j)=\sum\limits_{l=2}^{2^k-1}\left[\sum\limits_{\substack{v_i,v_j\in V_l\\uv\neq1}}d(v_i,v_j)-\sum\limits_{\substack{v_i,v_j\in V_1\\uv=1}}d(v_i,v_j)\right]$\\
				
				$=\left[2\{(\phi(n)-1)+(\phi(n)-2)+...+1\}
				-\left(\frac{\phi(n)-r}{2}\right)\right](2^k-2)$
				

				
				After Simplification, we get
				\begin{equation}\label{eq5}
					S_2=(2^{k-1}-1)\left(2\phi(n)^2-3\phi(n)+r\right).
				\end{equation}
				
				
				\item 	$S_3=\sum\limits_{\substack{v_p\in V_1\\v_q\in V_q\\q \neq 1}}d(v_p,v_q).$
				
				Since $V_q=\{(\mathfrak{e}_q,u):\mathfrak{e}_q \in Id({\mathbb{Z}_n})^*, u \in U(\mathbb{Z}_n)\}=\{v_1=(\mathfrak{e}_q,1),v_2=(\mathfrak{e}_q,u_2), v_3=(\mathfrak{e}_q,u_3),...,v_r=(\mathfrak{e}_q,u_r),v_{r+1}=(\mathfrak{e}_q,u_{r+1}),...,v_{\phi(n)}=(\mathfrak{e}_q,u_{\phi(n)})\}$ and
				$V_1=\{(1,u):1 \in Id(\mathbb{Z}_n)^*, u \in U(\mathbb{Z}_n)\}=\{v_1=(1,1),v_2=(1,u_2), v_3=(1,u_3),...,v_r=(1,u_r),v_{r+1}=(1,u_{r+1}),...,v_{\phi(n)}=(1,u_{\phi(n)})\}$. Let $v_i=(1,u)\in V_1$ and $v_j=(\mathfrak{e}_q,v)\in V_q$, $q\neq 1$, by Lemma \ref{Lemma 1},
				
				\begin{center}
					$ d(v_i,v_j)=\begin{cases}
						1 & uv=1\\
						2 & uv\neq 1.
					\end{cases} $ \\
				\end{center}
				
				There are $\phi(n)$ and $\phi(n)-1$ pairs of $v_i=(1,u)\in V_1$ and $v_j=(\mathfrak{e}_q,v)\in V_l$, $q\neq 1$, such that $ d(v_i,v_j)=1$ and $ d(v_i,v_j)=2$ respectively. So that, for $q=2,3,4,...,2^k-1$.\\
				
				$\sum\limits_{\substack{v_p\in V_1\\v_q\in V_q\\q \neq 1}}d(v_p,v_q)= (2^k-2)\left[\phi(n)+\left(2(\phi(n)-1)\right)\phi(n)\right]$.
				After simplification, we get 
				\begin{equation}\label{eq6}
					S_3=\sum\limits_{\substack{v_p\in V_1\\v_q\in V_q\\q \neq 1}}d(v_p,v_q)=(2^k-2)\left[2\phi(n)^2-\phi(n)\right]
				\end{equation}
				
				\item 	$S_4=\sum\limits_{\substack{v_p\in V_p\\v_q\in V_q\\p\neq q\neq 1}}d(v_p,v_q).$\\
				
				$\sum\limits_{\substack{v_p\in V_p\\v_q\in V_q\\p\neq q\neq 1}}d(v_p,v_q)=\sum\limits_{\substack{e,f\in Id(\mathbb{Z}_n)^*\\e\neq f\\u,v\in U(\mathbb{Z}_n)}}d((e,u),(f,v))$
				
				Since $V_p=\{(\mathfrak{e}_p,u):1 \in Id(\mathbb{Z}_n)^*, u \in U(\mathbb{Z}_n)\}=\{v_1=(\mathfrak{e}_p,1),v_2=(\mathfrak{e}_p,u_2), v_3=(\mathfrak{e}_p,u_3),...,v_r=(\mathfrak{e}_p,u_r),v_{r+1}=(\mathfrak{e}_p,u_{r+1}),...,v_{\phi(n)}=(\mathfrak{e}_p,u_{\phi(n)})\}$ and $V_q=\{(\mathfrak{e}_q,u):\mathfrak{e}_q \in Id({\mathbb{Z}_n})^*, u \in U(\mathbb{Z}_n)\}=\{v_1=(\mathfrak{e}_q,1),v_2=(\mathfrak{e}_q,u_2), v_3=(\mathfrak{e}_q,u_3),...,v_r=(\mathfrak{e}_q,u_r),v_{r+1}=(\mathfrak{e}_q,u_{r+1}),...,v_{\phi(n)}=(\mathfrak{e}_q,u_{\phi(n)})\}$.
				
				We write $S_4$ in explicit form as:\\
				\begin{equation}\label{eq7}
					S_4=\sum\limits_{\substack{(e,u)\in V_p\\(f,v)\in V_q\\p\neq q\neq 1}}d((e,u),(f,v))=\sum\limits_{\substack{(e,u)\in V_p\\(f,v)\in V_q\\ef=0\\p\neq q\neq1}}d((e,u),(f,v))+\sum\limits_{\substack{(e,u)\in V_p\\(f,v)\in V_q\\ef\neq0, uv=1\\p\neq q\neq 1}}d((e,u),(f,v))+\sum\limits_{\substack{(e,u)\in V_p\\(f,v)\in V_q\\ef\neq0, uv\neq1\\p\neq q\neq 1}}d((e,u),(f,v))
				\end{equation}
				Set $T_1=\sum\limits_{\substack{(e,u)\in V_p\\(f,v)\in V_q\\ef=0\\p\neq q\neq1}}d((e,u),(f,v))$, $T_2=\sum\limits_{\substack{(e,u)\in V_p\\(f,v)\in V_q\\ef\neq0, uv=1\\p\neq q\neq 1}}d((e,u),(f,v))$,  $T_3=\sum\limits_{\substack{(e,u)\in V_p\\(f,v)\in V_q\\ef\neq0, uv\neq1\\p\neq q\neq 1}}d((e,u),(f,v))$
				
				Now, we determine value of each summation $T_1$, $T_2$ and $T_3$ one by one as follows:
				\begin{enumerate}[label=(\roman*)]

					\item 	$T_1=\sum\limits_{\substack{(e,u)\in V_p\\(f,v)\in V_q\\ef=0\\p\neq q\neq1}}d((e,u),(f,v))$. 
					
					There are $\frac{2^k-2}{2}$ pairs of vertex set $V_p$ and $V_q$ such that for $v_p=(e,u)\in V_p$ and $v_q=(f,u)\in V_q$, $ef=1$. By Lemma \ref{Lemma 1} $d(v_p,v_q)=1$. So that, 
					
					\begin{equation}\label{eq8}
						T_1=\sum\limits_{\substack{(e,u)\in V_p\\(f,v)\in V_q\\ef=0\\p\neq q\neq1}}d((e,u),(f,v))=\phi(n)\left(\frac{2^k-2}{2}\right).
					\end{equation}
					
					\item 	$T_2=\sum\limits_{\substack{(e,u)\in V_p\\(f,v)\in V_q\\ef\neq0, uv=1\\p\neq q\neq 1}}d((e,u),(f,v)).$
					
					By lemma \ref{Lemma 1}, $d((e,u),(f,v))=\begin{cases}
						1 & ef=0\text{ or }uv=1\\
						2 & ef\neq 0 \text{ and } uv\neq1.
					\end{cases}$
					
					For each $(\mathfrak{e}_{2m},u_i) \in V_{2m}$ and $(1-\mathfrak{e}_{2m},u_j)\in V_{2m+1}$, $m=1,2,...,\frac{2^{k}-4}{2}$, there is unique $(f,u)$ in all $V_q$, $q>(2m+1)$ such that $(\mathfrak{e}_{2m},u_i)$ and $(1-\mathfrak{e}_{2m},u_j)$ are adjacent to $(f,u)$. Therefore, $d((\mathfrak{e}_{2m},u_i),(f,u))=d((1-\mathfrak{e}_{2m},u_i),(f,u))=1$ for all $i=1,2,...,\phi(n)-1,\phi(n).$ There are $\phi(n)$ vertices in each vertex set $V_p$ and $2^{k-1}-1$ pair of vertex sets. So that,
					
					$T_2=\sum\limits_{\substack{(e,u)\in V_p\\(f,v)\in V_q\\ef\neq0, uv=1\\p\neq q\neq 1}}d((e,u),(f,v))=2\phi(n)\sum\limits_{i=1}^{2^{k-1}-1}(2^k-2(i+1)).$
					
					After simplification, we get
					\begin{equation}\label{eq9} 
						T_2=2\phi(n)\left[2^{2k-1}-2^{k+1}-(2^{k-2}-1)(2^k+2)\right].
					\end{equation}
					\item 	$T_3=\sum\limits_{\substack{(e,u)\in V_p\\(f,v)\in V_q\\ef\neq0, uv\neq1\\p\neq q\neq 1}}d((e,u),(f,v))$
					
					By the similar argument for $d((\mathfrak{e}_{2m},u_i),(f,u))=d((1-\mathfrak{e}_{2m},u_i),(f,u))=2$, we apply multiplication principal. So that,
					
					$T_3=\sum\limits_{\substack{(e,u)\in V_p\\(f,v)\in V_q\\ef\neq0, uv\neq1\\p\neq q\neq 1}}d((e,u),(f,v))=4(\phi(n)-1)(\phi(n))\sum\limits_{i=1}^{2^{k-1}-1}(2^k-2(i+1)).$
					
					After simplification, we get
					\begin{equation}\label{eq10}
						T_3=(4\phi(n)^2-4\phi(n))\left[2^{2k-1}-2^{k+1}-(2^{k-2}-1)(2^k+2)\right]
					\end{equation}
				\end{enumerate}
				On adding $T_1$, $T_2$ and $T_3$, we get
				\begin{equation}\label{eq11}
					S_4=\sum\limits_{\substack{(e,u)\in V_p\\(f,v)\in V_q\\p\neq q\neq 1}}d((e,u),(f,v))=\phi(n)^2\left(\frac{2^k-2}{2}\right)+2\left[2\phi(n)^2-\phi(n)\right]\left[2^{2k-1}-2^{k+1}-(2^{k-2}-1)(2^k+2)\right]. 
				\end{equation}
			\end{enumerate}	
			On adding value of $S_1$, $S_2$, $S_3$ and $S_4$, we get Wiener index of clean of ring $\mathbb{Z}_n$ as 
			
			$W(Cl_2(\mathbb{Z}_n))=\frac{1}{2}\left[\phi(n)^2\left(2(2^{2k})-5(2^k)+5\right)-\phi(n)\left(2^{2k}-2^k+3\right)+2^kr\right].$
			
		\end{proof}
		\begin{corollary}
			Let $p$, $q$ and $r$ are distinct primes. Then

			\begin{enumerate}
				
				\item $W(Cl_2(Z_{pq}))=\begin{cases}
					\frac{17\left((p-1)(q-1)\right)^2-15(p-1)(q-1)+16 }{2} & \text{ if } p\text{ and } q\neq 2\\
					
					\frac{17\left((p-1)(q-1)\right)^2-15(p-1)(q-1)+8 }{2} & \text{ if } p\text{ or } q= 2
				\end{cases} $
				
				\item $W(Cl_2(Z_{pqr}))=\begin{cases}
					\frac{93\left((p-1)(q-1)(r-1)\right)^2-59(p-1)(q-1)(r-1)+64 }{2} & \text{ if } p\text{ and } q\neq 2\\
					
					\frac{93\left((p-1)(q-1)(r-1)\right)^2-59(p-1)(q-1)(r-1)+32 }{2} & \text{ if } p \text{ or } q= 2.\\
				\end{cases} $
			\end{enumerate}
		\end{corollary}
		
		Figure $1$  is an illustration of the above theorem . In this graph, we have partitioned the vertex set $ V(Cl_2(\mathbb{Z}_{pq})) $ in to three disjoint sets as $ V_1= \{(1,u):1\in Id(\mathbb{Z}_{pq}), u\in U(\mathbb{Z}_{pq})\}$, $ V_2= \{(e,u):e\in Id(\mathbb{Z}_{pq}), u\in U(\mathbb{Z}_{pq})\}$ and $ V_3= \{(1-e,u):1-e\in Id(\mathbb{Z}_{pq}), u\in U(\mathbb{Z}_{pq})\}$. 
		\begin{center}
			
			\begin{tikzpicture}\label{Graph 1}
				\draw[rotate=90] (-1,0) ellipse (150pt and 50pt);
				\draw (0,-6.5) node[below]{$V_3$};
				\fill (0,3) circle (2pt) node[right] {$w_1$}; 
				\fill (0,2) circle (2pt) node[right] {$w_2$}; 
				\fill (0,1) circle (2pt) node[right] {$w_3$}; 
				\fill (0,0) circle (2pt) node[right] {$w_4$};
				\fill (0,-1) circle (2pt) node[right] {$w_5$};
				\fill (0,-2) circle (2pt) node[right] {$w_6$};
				\fill (0,-2.5) circle (1pt); 
				\fill (0,-3) circle (1pt);
				\fill (0,-3.5) circle (1pt);
				\fill (0,-4) circle (2pt) node[right] {$w_{\phi(pq)-1}$};
				\fill (0,-5) circle (2pt) node[right] {$w_{\phi(pq)}$};
				
				\draw[rotate=90] (-1,4) ellipse (150pt and 50pt);
				\draw (-4,-7.25) node[below]{Figure 1: $Cl_2(\mathbb{Z}_{pq})$};
				\draw (-4,-6.5) node[below]{$V_2$};
				\fill (-4,3) circle (2pt) node[below left] {$v_1$}; 
				\fill (-4,2) circle (2pt) node[below left] {$v_2$}; 
				\fill (-4,1) circle (2pt) node[below left] {$v_3$}; 
				\fill (-4,0) circle (2pt) node[below left] {$v_4$};
				\fill (-4,-1) circle (2pt) node[below left] {$v_5$};
				\fill (-4,-2) circle (2pt) node[below left] {$v_6$};
				\fill (-4,-2.5) circle (1pt); 
				\fill (-4,-3) circle (1pt);
				\fill (-4,-3.5) circle (1pt);
				\fill (-4,-4) circle (2pt) node[below left] {$v_{\phi(pq)-1}$};
				\fill (-4,-5) circle (2pt) node[below left] {$v_{\phi(pq)}$};
				
				\draw[rotate=90] (-1,8) ellipse (150pt and 50pt);
				\draw (-8,-6.5) node[below]{$V_1$};
				\fill (-8,3) circle (2pt) node[left] {$u_1$}; 
				\fill (-8,2) circle (2pt) node[left] {$u_2$}; 
				\fill (-8,1) circle (2pt) node[left] {$u_3$}; 
				\fill (-8,0) circle (2pt) node[left] {$u_4$};
				\fill (-8,-1) circle (2pt) node[left] {$u_5$};
				\fill (-8,-2) circle (2pt) node[left] {$u_6$};
				\fill (-8,-2.5) circle (1pt); 
				\fill (-8,-3) circle (1pt);
				\fill (-8,-3.5) circle (1pt);
				\fill (-8,-4) circle (2pt) node[left] {$u_{\phi(pq)-1}$};
				\fill (-8,-5) circle (2pt) node[left] {$u_{\phi(pq)}$};
				
				\draw [black] (-8,3) to (-4,3) to (0,3) (-8,3) to[out=15,in=165] (0,3) (-4,3) to (0,2) (-4,3) to (0,1) (-4,3) to (0,0) (-4,3) to (0,-1) (-4,3) to (0,-2) (-4,3) to (0,-4) (-4,3) to (0,-5) ;
				\draw [black] (-8,2) to (-4,2) to (0,2) (-8,2) to[out=15,in=165] (0,2) (-4,2) to (0,1) (-4,2) to (0,2)(-4,2) to (0,3) (-4,2) to (0,0) (-4,2) to (0,-1) (-4,2) to (0,-2) (-4,2) to (0,-4) (-4,2) to (0,-5) ;
				\draw [black] (-8,1) to (-4,1) to (0,1) (-8,1) to[out=15,in=165] (0,1) (-4,1) to (0,1) (-4,1) to (0,2)(-4,1) to (0,3) (-4,1) to (0,0) (-4,1) to (0,-1) (-4,2) to (0,-2) (-4,2) to (0,-4) (-4,2) to (0,-5) ;
				\draw [black] (-8,0) to (-4,0) to (0,0) (-8,0) to[out=15,in=165] (0,0) (-4,0) to (0,0) (-4,0) to (0,2)(-4,0) to (0,3) (-4,0) to (0,1) (-4,0) to (0,-1) (-4,0)to (0,-2) (-4,0) to (0,-4) (-4,0) to (0,-5) ;
				\draw [black] (-8,-2) to (-8,-1)to (-4,-2) to (0,-2) (-8,-1) to[out=15,in=160] (0,-2) (-4,-1)to (-4,-2) (0,-1)to(0,-2) (-4,-1) to (0,0) (-4,-1) to (0,2)(-4,-1) to (0,3) (-4,-1) to (0,1) (-4,-1) to (0,-1) (-4,-1) to (0,-2) (-4,-1) to (0,-4) (-4,-1) to (0,-5);
				\draw[black] (-8,-2) to (-4,-1) (-8,-2) to [out=-15,in=-160](0,-1) (-4,-2)to(0,3) (-4,-2)to(0,2) (-4,-2) to (0,1) (-4,-2) to (0,0) (-4,-2) to (0,-1) (-4,-2) to (0,-2) (-4,-2) to (0,-4) (-4,-2) to (0,-5);
				\draw[black] (-8,-5) to (-8,-4)to(-4,-5)to(-4,-4)to(0,3) (-8,-4) to[out=15,in=160] (0,-5) (-8,-5)to[out=-15,in=-160](0,-4)to(0,-5)(-8,-5)to(-4,-4)to(0,3)(-4,-4)to(0,2)(-4,-4)to(0,1)(-4,-4)to(0,0)(-4,-4)to(0,-1)(-4,-4)to(0,-2)(-4,-4)to(0,-4)(-4,-4)to(0,-5) (-4,-5)to(0,3)(-4,-5)to(0,2)(-4,-5)to(0,1)(-4,-5)to(0,0)(-4,-5)to(0,-1)(-4,-5)to(0,-2)(-4,-5)to(0,-4)(-4,-5)to(0,-5);
			\end{tikzpicture}
			
		\end{center}
		\begin{example}
			Consider ring $\mathbb{Z}_{15}$, then $ Id(\mathbb{Z}_{15})^*=\{1,6,10\} $ and $U(\mathbb{Z}_{15})=\{1,2,4,7,8,11,13,14\}$, $V(Cl_2(\mathbb{Z}_{15}))=\{(1,1),(1,2),...,(1,14),(6,1),(6,2),...,(6,14),...,(10,14)\}$. Here $ n=15 $, $ \phi(15)= 8 $, and $ k=2 $. So that,
			
			$ W(Cl_2(\mathbb{Z}_{15}))=\frac{1}{2}[8^2(2(2^{4})-5(2^2)+5)-8(2^4-2^2+3)+2^22^2] =492$.
			
		\end{example}
		\begin{center}
			\scalebox{0.6}{
				\begin{tikzpicture}\label{Figure 2}
					\fill (4,0) circle (2pt) node[right]{(1,1)};
					\fill (3.864,1.035) circle (2pt)node[right]{(1,2)};
					\fill (3.464,2) circle (2pt)node[right]{(1,4)}; 
					\fill (2.828,2.828) circle (2pt)node[right]{(1,7)};
					\fill (2,3.464) circle (2pt)node[above right]{(1,8)}; 
					\fill (1.035,3.864) circle (2pt) node[above right]{(1,11)}; 
					\fill (0,4) circle (2pt)node[above]{(1,13)}; 
					\fill (-1.035,3.864) circle (2pt)node[above left]{(1,14)};
					\fill (-2,3.464) circle (2pt)node[above left]{(6,1)};
					\fill (-2.828,2.828) circle (2pt)node[left]{(6,2)};  
					\fill (-3.464,2) circle (2pt)node[left]{(6,4)};
					\fill (-3.864,1.035) circle (2pt)node[left]{(6,7)};
					\fill (-4,0) circle (2pt)node[left]{(6,8)}; 
					\fill (-3.864,-1.035) circle (2pt)node[left]{(6,11)}; 
					\fill (-3.464,-2) circle (2pt)node[left]{(6,13)};
					\fill (-2.828,-2.828) circle (2pt)node[left]{(6,14)};  
					\fill (-2,-3.464) circle (2pt)node[below left]{(10,1)}; 
					\fill (-1.035,-3.864) circle (2pt)node[below]{(10,2)}; 
					\fill (0,-4) circle (2pt)node[below]{(10,4)};
					\draw (0,-4.5) node[below]{Figure 2: $Cl_2(\mathbb{Z}_{15})$}; 
					\fill (1.035,-3.864) circle (2pt)node[below]{(10,7)};
					\fill (2,-3.464) circle (2pt)node[below right]{(10,8)};
					\fill (2.828,-2.828) circle (2pt)node[below right]{(10,11)};  
					\fill (3.464,-2) circle (2pt)node[right]{(10,13)};
					\fill (3.864,-1.035) circle (2pt)node[right]{(10,14)};
					\draw (4,0)--(-2,3.464);
					\draw (4,0)--(-2,-3.464);
					\draw (3.864,1.035)--(2,3.464);
					\draw (3.864,1.035)--(-4,0);
					\draw (3.864,1.035)--(2,-3.464);
					\draw (3.464,2)--(-3.464,2);
					\draw (3.464,2)--(0,-4);
					\draw (2.828,2.828)--(0,4);
					\draw (2.828,2.828)--(-3.464,-2);
					\draw (2.828,2.828)--(3.464,-2);
					\draw (2,3.464)--(-2.828,2.828);
					\draw (2,3.464)--(-1.035,-3.864);
					\draw (1.035,3.864)--(-3.864,-1.035);
					\draw (1.035,3.864)--(2.828,-2.828);
					\draw (0,4)--(-3.864,1.035);
					\draw (0,4)--(1.035,-3.864);
					\draw (-1.035,3.864)--(-2.828,-2.828);
					\draw (-1.035,3.864)--(-2.828,-2.828);
					\draw (-1.035,3.864)--(3.864,-1.035);
					\draw (-2,3.464)--(-2,-3.464);
					\draw (-2.828,2.828)--(-4,0);
					\draw (-2.828,2.828)--(-2,-3.464);
					\draw (-2.828,2.828)--(-1.035,-3.864);
					\draw (-2.828,2.828)--(0,-4);
					\draw (-2.828,2.828)--(1.035,-3.864);
					\draw (-2.828,2.828)--(2,-3.464);
					\draw (-2.828,2.828)--(2.828,-2.828);
					\draw (-2.828,2.828)--(3.464,-2);
					\draw (-2.828,2.828)--(3.864,-1.035);
					\draw (-3.464,2)--(-2,-3.464);
					\draw (-3.464,2)--(-1.035,-3.864);
					\draw (-3.464,2)--(0,-4);
					\draw (-3.464,2)--(1.035,-3.864);
					\draw (-3.464,2)--(-2,-3.464);
					\draw (-3.464,2)--(2,-3.464);
					\draw (-3.464,2)--(2.828,-2.828);
					\draw (-3.464,2)--(3.464,-2);
					\draw (-3.464,2)--(3.864,-1.035);
					\draw (-3.864,1.035)--(-3.464,-2);
					\draw (-3.864,1.035)--(-1.035,-3.864);
					\draw (-3.864,1.035)--(0,-4);
					\draw (-3.864,1.035)--(1.035,-3.864);
					\draw (-3.864,1.035)--(-2,-3.464);
					\draw (-3.864,1.035)--(2,-3.464);
					\draw (-3.864,1.035)--(2.828,-2.828);
					\draw (-3.864,1.035)--(3.464,-2);
					\draw (-3.864,1.035)--(3.864,-1.035);
					\draw (-4,0)--(-1.035,-3.864);
					\draw (-4,0)--(0,-4);
					\draw (-4,0)--(1.035,-3.864);
					\draw (-4,0)--(-2,-3.464);
					\draw (-4,0)--(2,-3.464);
					\draw (-4,0)--(2.828,-2.828);
					\draw (-4,0)--(3.464,-2);
					\draw (-4,0)--(3.864,-1.035);
					\draw (-3.864,-1.035)--(-1.035,-3.864);
					\draw (-3.864,-1.035)--(0,-4);
					\draw (-3.864,-1.035)--(1.035,-3.864);
					\draw (-3.864,-1.035)--(-2,-3.464);
					\draw (-3.864,-1.035)--(2,-3.464);
					\draw (-3.864,-1.035)--(2.828,-2.828);
					\draw (-3.864,-1.035)--(3.464,-2);
					\draw (-3.864,-1.035)--(3.864,-1.035);
					\draw (-3.464,-2)--(-1.035,-3.864);
					\draw (-3.464,-2)--(0,-4);
					\draw (-3.464,-2)--(1.035,-3.864);
					\draw (-3.464,-2)--(-2,-3.464);
					\draw (-3.464,-2)--(2,-3.464);
					\draw (-3.464,-2)--(2.828,-2.828);
					\draw (-3.464,-2)--(3.464,-2);
					\draw (-3.464,-2)--(3.864,-1.035);
					\draw (-2.828,-2.828)--(-1.035,-3.864);
					\draw (-2.828,-2.828)--(0,-4);
					\draw (-2.828,-2.828)--(1.035,-3.864);
					\draw (-2.828,-2.828)--(-2,-3.464);
					\draw (-2.828,-2.828)--(2,-3.464);
					\draw (-2.828,-2.828)--(2.828,-2.828);
					\draw (-2.828,-2.828)--(3.464,-2);
					\draw (-2.828,-2.828)--(3.864,-1.035);
					\draw (-1.035,-3.864)--(2,-3.464);
					\draw (1.035,-3.864)--(3.464,-2);
					
				\end{tikzpicture}
			}
		\end{center}

		The close form to calculate the Wiener index of $Cl_2(\mathbb{Z}_n)$ is given in Table \ref{Table 1}. From this table, we can calculate the Wiener index of $ Cl_2(\mathbb{Z}_n) $. In Table \ref{Table 1}, all $ p_i $'s are distinct primes and not equal to 2. Here $ x= \phi(n)$.
		\begin{table}[h!]
			\caption{Wiener index of $Cl_2(\mathbb{Z}_n)$}\label{Table 1}
			\begin{center}

				\begin{tabular}{|c|c|c|c|}
					\hline
					$ n $ &$ k $&$ r $&$W(Cl_2(\mathbb{Z}_n))$ \\\hline
					$2{p_1}^{\alpha_1}$&$ 2 $&$  2$&$ \frac{1}{2}(17x^2-15x+8) $\\\hline
					$2{p_1}^{\alpha_1}{p_2}^{\alpha_2}$&$3$&$4$&$\frac{1}{2}(29x^2-59x+32)$\\\hline
					$2{p_1}^{\alpha_1}{p_2}^{\alpha_2}{p_3}^{\alpha_3}$&$4  $&$8 $&$\frac{1}{2}(181x^2-243x+128)$\\\hline
					$2{p_1}^{\alpha_1}{p_2}^{\alpha_2}{p_3}^{\alpha_3}{p_4}^{\alpha_4}$&$ 5 $&$ 16 $&$\frac{1}{2}(869x^2-995x+512)$\\\hline
					$2{p_1}^{\alpha_1}{p_2}^{\alpha_2}{p_3}^{\alpha_3}{p_4}^{\alpha_4}{p_5}^{\alpha_5}$&$6$&$32$&$\frac{1}{2}(3781x^2-4035x+2048)$\\\hline
					${p_1}^{\alpha_1}{p_2}^{\alpha_2}$&$2$&$4$&$\frac{1}{2}(x^2-15x+16)$\\\hline
					${p_1}^{\alpha_1}{p_2}^{\alpha_2}{p_3}^{\alpha_3}$&$3  $&$8  $&$  \frac{1}{2}(29x^2-59x+64)$\\\hline
					${p_1}^{\alpha_1}{p_2}^{\alpha_2}{p_3}^{\alpha_3}{p_4}^{\alpha_4}$&$4  $&$16  $&$  \frac{1}{2}(181x^2-243x+256)$\\\hline
					${p_1}^{\alpha_1}{p_2}^{\alpha_2}{p_3}^{\alpha_3}{p_4}^{\alpha_4}{p_5}^{\alpha_5}$&$  5$&$ 32 $&$  \frac{1}{2}(869x^2-995x+1024)$\\\hline
					${p_1}^{\alpha_1}{p_2}^{\alpha_2}{p_3}^{\alpha_3}{p_4}^{\alpha_4}{p_5}^{\alpha_5}{p_3}^{\alpha_6}$&$6$&$64  $&$ \frac{1}{2}(3781x^2-4035x+4096) $\\\hline

				\end{tabular}
			\end{center}
			
		\end{table}
		
		We can find Wiener index of clean graph of $\mathbb{Z}_n$ by plotting scattered diagram as follows:
		
		For example, we consider $ n=2{p_1}^{\alpha_1}$.
		
		\hspace{1cm}
		
		\pgfplotsset{ 
			standard/.style={
				axis line style = thick,
				trig format= rad,
				enlargelimits=false,
				axis x line = middle,
				axis y line = middle,
				enlarge x limits=0.15,
				enlarge y limits =0.15,
				every axis x label/.style={at={(current axis.right of origin)},anchor=north west},
				every axis y label/.style={at={(current axis.above origin)},anchor=south east},
				ticklabel style={font=\tiny,fill=white}
			}
		}
		\scalebox{0.9}{
			\begin{tikzpicture}\label{Graph 3}
				\begin{axis}[standard,
					xlabel={$ \phi(n) $},
					ylabel={$W(Cl_2(\mathbb{Z}_n))$},
					grid=both, 
					ytick={23,110,265,488,779,1138,1565},
					yticklabels={$ 23 $,$ 110 $,$ 265 $,$ 488 $,$ 779 $,$ 1138 $,$ 1565 $},
					xtick={2,4,6,8,10,12,14},
					xticklabels={$ 2 $,$ 4 $,$ 6 $,$ 8 $,$ 10 $,$ 12 $,$ 14 $},
					scatter/classes={a={black}, b={mark=o, draw=black}},
					xmin=-0.5, xmax=16, ymin=-1, ymax=1800]
					\node [anchor=center, label=south west:{$ O $}] at (axis cs:0,0){};
					\addplot [scatter,only marks, scatter src=explicit symbolic]coordinates {
						(2,23) [a]
						(4,110)[a]
						(6,265)[a]
						(8,488)[a]
						(10,779)[a]
						(12,1138)[a]
						(14,1565)[b]
					};	
					
				\end{axis}
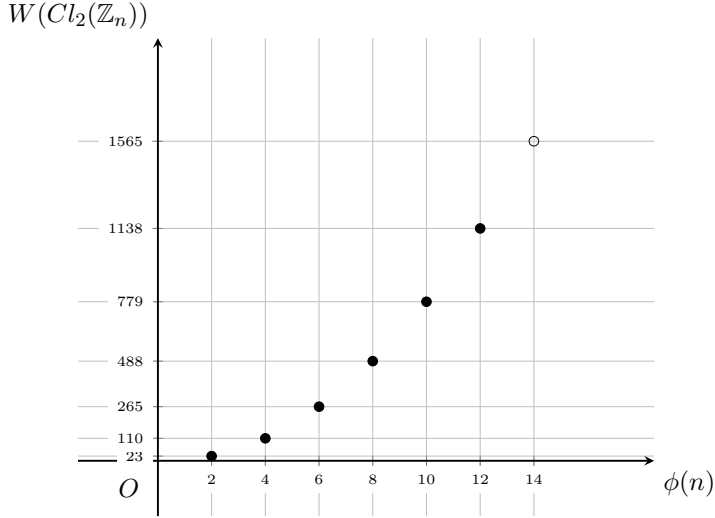
				\draw (4,-0.2) node[below]{Figure 3: Graph of $ \phi(n) $ versus $W(Cl_2(\mathbb{Z}_{2{p_1}^{\alpha_1}))}$ }; 
		\end{tikzpicture}}
		
		There does not exist $ n $ such that $ \phi(n)=14 $. Therefore, we have drawn circle at $ x=14 $. Table \ref{Table 2} gives the idea to understand the Figure $3$.
		
		\begin{table}[h!]
			\caption{Wiener index of $W(Cl_2(\mathbb{Z}_n))$}\label{Table 2}
			\begin{center}
				
				\begin{tabular}{|c|c|c|c|c|c|c|}
					\hline
					$ n $&$ 6 $&$ 10,12 $&$14,18  $&$ 20,24 $&$22  $&$ 26,36 $\\\hline
					$ \phi(n) $&$ 2 $&$ 4 $&$ 6 $&$8  $&$ 10 $&$12 $\\\hline
					$ W(Cl_2(\mathbb{Z}_n)) $&$ 23 $&$110  $&$ 265 $&$488  $&$  779$&$1138  $\\\hline

				\end{tabular}
			\end{center}
			
		\end{table}
		
		\begin{theorem}
			For $n=\prod\limits_{i=1}^{2^k-1}{p_i}^{\alpha_i}$,  $Cl_2(\mathbb{Z}_n)$ contains a perfect matching and $\mu(Cl_2(\mathbb{Z}_n))=\frac{\phi(n)(2^k-1)}{2}$. 
		\end{theorem}
		
		\begin{proof}
			Consider $2^k-1$ partition of the vertex set $V(Cl_2(\mathbb{Z}_n))$ as $V(Cl_2(\mathbb{Z}_n))=V_1\sqcup V_2\sqcup...\sqcup V_{2^k-1},$ where $V_{i}=\{v_{ji}=(e_i,u_j)/e_i\in Id(\mathbb{Z}_n)^*, u_j\in U(\mathbb{Z}_n)\}$ as shown in Figure $4$. By Definition \ref{Clean Graph}, we have following adjacency:
			\begin{enumerate}
				\item $(e_i,u_j)\sim (e_k,u_j),$ for $j=1,2,...,r-1,r.$
				\item  $(e_i,u_j)\sim (e_i,u_{j+1}),$ for $j=r+1,r+3,...,\phi(n)-3,\phi(n)-1$.
				\item $(e_i,u_j)\sim (e_k,u_{j+1}) \text{ and } (e_i,u_{j+1})\sim (e_k,u_{j}),$ for $i\neq k$ and $j=r+1,r+3,...,\phi(n)-3,\phi(n)-1$.
				\item $(e_i,u_j)\sim(e_{i+1},u_k)$ for $i=2,4,...,2^k-2.$
			\end{enumerate}
			
			\begin{center}
				\scalebox{0.65}{
					\begin{tikzpicture} \label{Graph 4}
						\draw (-9.3,-16) rectangle (-7.1,8);
						\draw (-8,-16) node[below]{$\mathbf {V_1}$};
						\fill (-8,7) circle (3pt) node[below] {$v_{11}$}; 
						\fill (-8,5) circle (3pt) node[below] {$v_{21}$}; 
						\fill (-8,3) circle (3pt) node[below] {$v_{31}$}; 
						\fill (-8,1) circle (3pt) node[below] {$v_{41}$};
						\fill (-8,0.5) circle (1pt); 
						\fill (-8,0) circle (1pt); 
						\fill (-8,-0.5) circle (1pt);
						\fill (-8,-1) circle (3pt) node[below] {$v_{(r-1)1}$};
						\fill (-8,-3) circle (3pt) node[below] {$v_{r1}$};
						\fill (-8,-5) circle (3pt) node[below] {$v_{(r+1)1}$};
						\fill (-8,-7) circle (3pt) node[below] {$v_{(r+2)1}$};
						\fill (-8,-9) circle (3pt) node[below] {$v_{(r+3)1}$};
						\fill (-8,-11) circle (3pt) node[below] {$v_{(r+4)1}$};
						\fill (-8,-11.5) circle (1pt);
						\fill (-8,-12) circle (1pt);
						\fill (-8,-12.5) circle (1pt);
						\fill (-8,-13) circle (3pt) node[below] {$v_{(\phi(n)-1)1}$};
						\fill (-8,-15) circle (3pt) node[below] {$v_{\phi(n)1}$};
						

						\draw (-6.5,-16) rectangle (-2.6,8);
						\draw (-5.5,-16) node[below]{$\mathbf {V_2}$};
						\fill (-5.5,7) circle (3pt) node[below] {$v_{12}$}; 
						\fill (-5.5,5) circle (3pt) node[below] {$v_{22}$}; 
						\fill (-5.5,3) circle (3pt) node[below] {$v_{32}$}; 
						\fill (-5.5,1) circle (3pt) node[below] {$v_{42}$};
						\fill (-5.5,0.5) circle (1pt); 
						\fill (-5.5,0) circle (1pt); 
						\fill (-5.5,-0.5) circle (1pt);
						\fill (-5.5,-1) circle (3pt) node[below] {$v_{(r-1)2}$};
						\fill (-5.5,-3) circle (3pt) node[below] {$v_{r2}$};
						\fill (-5.5,-5) circle (3pt) node[below] {$v_{(r+1)2}$};
						\fill (-5.5,-7) circle (3pt) node[below] {$v_{(r+2)2}$};
						\fill (-5.5,-9) circle (3pt) node[below] {$v_{(r+3)2}$};
						\fill (-5.5,-11) circle (3pt) node[below] {$v_{(r+4)2}$};
						\fill (-5.5,-11.5) circle (1pt);
						\fill (-5.5,-12) circle (1pt);
						\fill (-5.5,-12.5) circle (1pt);
						\fill (-5.5,-13) circle (3pt) node[below] {$v_{(\phi(pq)-1)2}$};
						\fill (-5.5,-15) circle (3pt) node[below] {$v_{\phi(pq)2}$};

						\draw (-3.6,-16) node[below]{$\mathbf {V_3}$};
						\fill (-3.6,7) circle (3pt) node[below] {$v_{13}$}; 
						\fill (-3.6,5) circle (3pt) node[below] {$v_{23}$}; 
						\fill (-3.6,3) circle (3pt) node[below] {$v_{33}$}; 
						\fill (-3.6,1) circle (3pt) node[below] {$v_{43}$};
						\fill (-3.6,0.5) circle (1pt); 
						\fill (-3.6,0) circle (1pt); 
						\fill (-3.6,-0.5) circle (1pt);
						\fill (-3.6,-1) circle (3pt) node[below] {$v_{(r-1)3}$};
						\fill (-3.6,-3) circle (3pt) node[below] {$v_{r3}$};
						\fill (-3.6,-5) circle (3pt) node[below] {$v_{(r+1)3}$};
						\fill (-3.6,-7) circle (3pt) node[below] {$v_{(r+2)3}$};
						\fill (-3.6,-9) circle (3pt) node[below] {$v_{(r+3)3}$};
						\fill (-3.6,-11) circle (3pt) node[below] {$v_{(r+4)3}$};
						\fill (-3.6,-11.5) circle (1pt);
						\fill (-3.6,-12) circle (1pt);
						\fill (-3.6,-12.5) circle (1pt);
						\fill (-3.6,-13) circle (3pt) node[below] {$v_{(\phi(n)-1)3}$};
						\fill (-3.6,-15) circle (3pt) node[below] {$v_{\phi(n)3}$};
						\draw (-2,-16)  rectangle (1.9,8);
						\draw (-1,-16) node[below]{$\mathbf {V_4}$};
						\fill (-1,7) circle (3pt) node[below] {$v_{14}$}; 
						\fill (-1,5) circle (3pt) node[below] {$v_{24}$}; 
						\fill (-1,3) circle (3pt) node[below] {$v_{34}$}; 
						\fill (-1,1) circle (3pt) node[below] {$v_{44}$};
						\fill (-1,0.5) circle (1pt); 
						\fill (-1,0) circle (1pt); 
						\fill (-1,-0.5) circle (1pt);
						\fill (-1,-1) circle (3pt) node[below] {$v_{(r-1)4}$};
						\fill (-1,-3) circle (3pt) node[below] {$v_{r4}$};
						\fill (-1,-5) circle (3pt) node[below] {$v_{(r+1)4}$};
						\fill (-1,-7) circle (3pt) node[below] {$v_{(r+2)4}$};
						\fill (-1,-9) circle (3pt) node[below] {$v_{(r+3)4}$};
						\fill (-1,-11) circle (3pt) node[below] {$v_{(r+4)4}$};
						\fill (-1,-11.5) circle (1pt);
						\fill (-1,-12) circle (1pt);
						\fill (-1,-12.5) circle (1pt);
						\fill (-1,-13) circle (3pt) node[below] {$v_{(\phi(n)-1)4}$};
						\fill (-1,-15) circle (3pt) node[below] {$v_{\phi(n)4}$};

						\fill  (2.9,-3) circle (1.5pt)(3.2,-3) circle (1.5pt)(3.5,-3)  circle (1.5pt);

						
						\draw(4.6,-16) rectangle (10.4,8);
						\draw (0.9,-16) node[below]{$\mathbf {V_{5}}$};
						\fill (0.9,7) circle (3pt) node[below] {$v_{15}$}; 
						\fill (0.9,5) circle (3pt) node[below] {$v_{25}$};
						\fill (0.9,3) circle (3pt) node[below] {$v_{35}$}; 
						\fill (0.9,1) circle (3pt) node[below] {$v_{45}$};
						\fill (0.9,0.5) circle (1pt); 
						\fill (0.9,0) circle (1pt); 
						\fill (0.9,-0.5) circle (1pt);
						\fill (0.9,-1) circle (3pt) node[below] {$v_{(r-1)5}$};
						\fill (0.9,-3) circle (3pt) node[below] {$v_{r5}$};
						\fill (0.9,-5) circle (3pt) node[below] {$v_{(r+1)5}$};
						\fill (0.9,-7) circle (3pt) node[below] {$v_{(r+2)5}$};
						\fill (0.9,-9) circle (3pt) node[below] {$v_{(r+3)5}$};
						\fill (0.9,-11) circle (3pt) node[below] {$v_{(r+4)5}$};
						\fill (0.9,-11.5) circle (1pt);
						\fill (0.9,-12) circle (1pt);
						\fill (0.9,-12.5) circle (1pt);
						\fill (0.9,-13) circle (3pt) node[below] {$v_{(\phi(n)-1)5}$};
						\fill (0.9,-15) circle (3pt) node[below] {$v_{\phi(n)5}$};
						
						
						\draw (6,-16) node[below]{$\mathbf {V_{(2^k-2)}}$};
						\fill (6,7) circle (3pt) node[below] {$v_{1{(2^k-2)}}$}; 
						\fill (6,5) circle (3pt) node[below] {$v_{2{(2^k-2)}}$};
						\fill (6,3) circle (3pt) node[below] {$v_{3(2^k-2)}$}; 
						\fill (6,1) circle (3pt) node[below] {$v_{4(2^k-2)}$};
						\fill (6,0.5) circle (1pt); 
						\fill (6,0) circle (1pt); 
						\fill (6,-0.5) circle (1pt);
						\fill (6,-1) circle (3pt) node[below] {$v_{(r-1)(2^k-2)}$};
						\fill (6,-3) circle (3pt) node[below] {$v_{r(2^k-2)}$};
						\fill (6,-5) circle (3pt) node[below] {$v_{(r+1)(2^k-2)}$};
						\fill (6,-7) circle (3pt) node[below] {$v_{(r+2)(2^k-2)}$};
						\fill (6,-9) circle (3pt) node[below] {$v_{(r+3)(2^k-2)}$};
						\fill (6,-11) circle (3pt) node[below] {$v_{(r+4)(2^k-2)}$};
						\fill (6,-11.5) circle (1pt);
						\fill (6,-12) circle (1pt);
						\fill (6,-12.5) circle (1pt);
						\fill (6,-13) circle (3pt) node[below] {$v_{(\phi(n)-1)(2^k-2)}$};
						\fill (6,-15) circle (3pt) node[below] {$v_{\phi(n)(2^k-2)}$};
						
						\draw (9,-16) node[below]{$\mathbf {V_{(2^k-1)}}$};
						\fill (9,7) circle (3pt) node[below] {$v_{1{(2^k-1)}}$}; 
						\fill (9,5) circle (3pt) node[below] {$v_{2{(2^k-1)}}$};
						\fill (9,3) circle (3pt) node[below] {$v_{3(2^k-1)}$}; 
						\fill (9,1) circle (3pt) node[below] {$v_{4(2^k-1)}$};
						\fill (9,0.5) circle (1pt); 
						\fill (9,0) circle (1pt); 
						\fill (9,-0.5) circle (1pt);
						\fill (9,-1) circle (3pt) node[below] {$v_{(r-1)(2^k-1)}$};
						\fill (9,-3) circle (3pt) node[below] {$v_{r(2^k-1)}$};
						\fill (9,-5) circle (3pt) node[below] {$v_{(r+1)(2^k-1)}$};
						\fill (9,-7) circle (3pt) node[below] {$v_{(r+2)(2^k-1)}$};
						\fill (9,-9) circle (2pt) node[below] {$v_{(r+3)(2^k-1)}$};
						\fill (9,-11) circle (3pt) node[below] {$v_{(r+4)(2^k-1)}$};
						
						\fill (9,-11.5) circle (1pt);
						\fill (9,-12) circle (1pt);
						\fill (9,-12.5) circle (1pt);
						\fill (9,-13) circle (3pt) node[below] {$v_{(\phi(n)-1)(2^k-1)}$};
						\fill (9,-15) circle (3pt) node[below] {$v_{\phi(n)(2^k-1)}$};
						
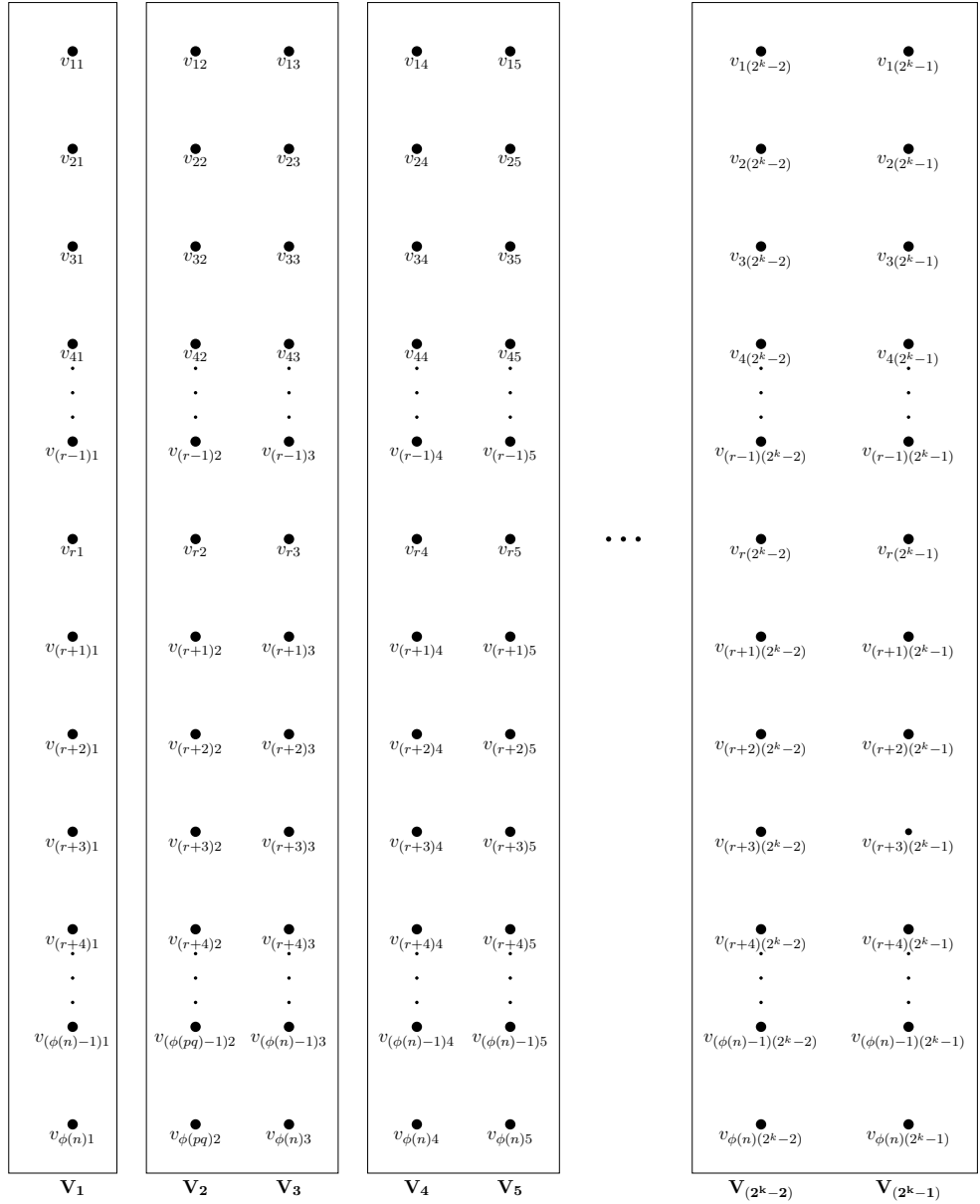
						\draw (0,-18) node[below]{\textbf{Figure 4:} Partition of $V(\mathbf{Cl_2(\mathbb{Z}_n)})$};

				\end{tikzpicture}}
			\end{center}
			
			Now, in Figure $5$, we draw a perfect matching only. We can se all the vertices are saturated, see Figure $5$. Thus, $Cl_2(\mathbb{Z}_n)$ has a perfect matching and $\mu(Cl_2(\mathbb{Z}_n))=\frac{\phi(n)(2^k-1)}{2}$

			\begin{center}\scalebox{0.65}{
					
					\begin{tikzpicture}\label{Figure 5}
						
						\draw (-8,-16) node[below]{$V_1$};
						\fill (-8,7) circle (3pt) node[below left] {$v_{11}$}; 
						\fill (-8,5) circle (3pt) node[below left] {$v_{21}$}; 
						\fill (-8,3) circle (3pt) node[below left] {$v_{31}$}; 
						\fill (-8,1) circle (3pt) node[below left] {$v_{41}$};
						\fill (-8,0.5) circle (1pt); 
						\fill (-8,0) circle (1pt); 
						\fill (-8,-0.5) circle (1pt);
						\fill (-8,-1) circle (3pt) node[below left] {$v_{(r-1)1}$};
						\fill (-8,-3) circle (3pt) node[below left] {$v_{r1}$};
						\fill (-8,-5) circle (3pt) node[below left] {$v_{(r+1)1}$};
						\fill (-8,-7) circle (3pt) node[below left] {$v_{(r+2)1}$};
						
						\fill (-8,-9) circle (3pt) node[below left] {$v_{(r+3)1}$};
						\fill (-8,-11) circle (3pt) node[below left] {$v_{(r+4)1}$};
						
						\fill (-8,-11.5) circle (1pt);
						\fill (-8,-12) circle (1pt);
						\fill (-8,-12.5) circle (1pt);
						
						\fill (-8,-13) circle (3pt) node[below left] {$v_{(\phi(n)-1)1}$};
						\fill (-8,-15) circle (3pt) node[below left] {$v_{\phi(n)1}$};
						
						
						\draw (-5.5,-16) node[below]{$V_2$};
						\fill (-5.5,7) circle (3pt) node[below] {$v_{12}$}; 
						\fill (-5.5,5) circle (3pt) node[below] {$v_{22}$}; 
						\fill (-5.5,3) circle (3pt) node[below] {$v_{32}$}; 
						\fill (-5.5,1) circle (3pt) node[below] {$v_{42}$};
						\fill (-5.5,0.5) circle (1pt); 
						\fill (-5.5,0) circle (1pt); 
						\fill (-5.5,-0.5) circle (1pt);
						\fill (-5.5,-1) circle (3pt) node[below] {$v_{(r-1)2}$};
						\fill (-5.5,-3) circle (3pt) node[below] {$v_{r2}$};
						\fill (-5.5,-5) circle (3pt) node[left] {$v_{(r+1)2}$};
						\fill (-5.5,-7) circle (3pt) node[below] {$v_{(r+2)2}$};
						\fill (-5.5,-9) circle (3pt) node[left] {$v_{(r+3)2}$};
						\fill (-5.5,-11) circle (3pt) node[below] {$v_{(r+4)2}$};
						\fill (-5.5,-11.5) circle (1pt);
						\fill (-5.5,-12) circle (1pt);
						\fill (-5.5,-12.5) circle (1pt);
						\fill (-5.5,-13) circle (3pt) node[left] {$v_{(\phi(n)-1)2}$};
						\fill (-5.5,-15) circle (3pt) node[below] {$v_{\phi(n)2}$};
						
						\draw (-2.5,-16) node[below]{$V_3$};
						\fill (-2.5,7) circle (3pt) node[below] {$v_{13}$}; 
						\fill (-2.5,5) circle (3pt) node[below] {$v_{23}$}; 
						\fill (-2.5,3) circle (3pt) node[below] {$v_{33}$}; 
						\fill (-2.5,1) circle (3pt) node[below] {$v_{43}$};
						\fill (-2.5,0.5) circle (1pt); 
						\fill (-2.5,0) circle (1pt); 
						\fill (-2.5,-0.5) circle (1pt);
						\fill (-2.5,-1) circle (3pt) node[below] {$v_{(r-1)3}$};
						\fill (-2.5,-3) circle (3pt) node[below] {$v_{r3}$};
						\fill (-2.5,-5) circle (3pt) node[below right] {$v_{(r+1)3}$};
						\fill (-2.5,-7) circle (3pt) node[below] {$v_{(r+2)3}$};
						\fill (-2.5,-9) circle (3pt) node[below right] {$v_{(r+3)3}$};
						\fill (-2.5,-11) circle (3pt) node[below] {$v_{(r+4)3}$};
						\fill (-2.5,-11.5) circle (1pt);
						\fill (-2.5,-12) circle (1pt);
						\fill (-2.5,-12.5) circle (1pt);
						\fill (-2.5,-13) circle (3pt) node[left] {$v_{(\phi(n)-1)3}$};
						\fill (-2.5,-15) circle (3pt) node[below] {$v_{\phi(n)3}$};
						
						\draw (0,-16) node[below]{$V_4$};
						\fill (0,7) circle (3pt) node[below] {$v_{14}$}; 
						\fill (0,5) circle (3pt) node[below] {$v_{24}$}; 
						\fill (0,3) circle (3pt) node[below] {$v_{34}$}; 
						\fill (0,1) circle (3pt) node[below] {$v_{44}$};
						\fill (0,0.5) circle (1pt); 
						\fill (0,0) circle (1pt); 
						\fill (0,-0.5) circle (1pt);
						\fill (0,-1) circle (3pt) node[below] {$v_{(r-1)4}$};
						\fill (0,-3) circle (3pt) node[below] {$v_{r4}$};
						\fill (0,-5) circle (3pt) node[below right] {$v_{(r+1)4}$};
						\fill (0,-7) circle (3pt) node[below] {$v_{(r+2)4}$};
						\fill (0,-9) circle (3pt) node[below right] {$v_{(r+3)4}$};
						\fill (0,-11) circle (3pt) node[below] {$v_{(r+4)4}$};
						\fill (0,-11.5) circle (1pt);
						\fill (0,-12) circle (1pt);
						\fill (0,-12.5) circle (1pt);
						\fill (0,-13) circle (3pt) node[below right] {$v_{(\phi(n)-1)4}$};
						\fill (0,-15) circle (3pt) node[below] {$v_{\phi(n)4}$};

						\fill  (1,7) circle (1pt)(1.3,7) circle (1pt)(1.6,7)  circle (1pt);
						\fill  (1,5) circle (1pt)(1.3,5) circle (1pt)(1.6,5)  circle (1pt);
						\fill  (1,3) circle (1pt)(1.3,3) circle (1pt)(1.6,3)  circle (1pt);
						\fill  (1,1) circle (1pt)(1.3,1) circle (1pt)(1.6,1)  circle (1pt);
						\fill  (1,-1) circle (1pt)(1.3,-1) circle (1pt)(1.6,-1)  circle (1pt);
						\fill  (1,-3) circle (1pt)(1.3,-3) circle (1pt)(1.6,-3)  circle (1pt);
						\fill  (1,-5) circle (1pt)(1.3,-5) circle (1pt)(1.6,-5)  circle (1pt);
						\fill  (1,-7) circle (1pt)(1.3,-7) circle (1pt)(1.6,-7)  circle (1pt);
						\fill  (1,-9) circle (1pt)(1.3,-9) circle (1pt)(1.6,-9)  circle (1pt);
						\fill  (1,-11) circle (1pt)(1.3,-11) circle (1pt)(1.6,-11)  circle (1pt);
						\fill  (1,-13) circle (1pt)(1.3,-13) circle (1pt)(1.6,-13)  circle (1pt);
						\fill  (1,-15) circle (1pt)(1.3,-15) circle (1pt)(1.6,-15)  circle (1pt);
						
						
						\draw (3,-16) node[below]{$V_{(2^k-3)}$};
						\fill (3,7) circle (3pt) node[below] {$v_{1{(2^k-3)}}$}; 
						\fill (3,5) circle (3pt) node[below] {$v_{2{(2^k-3)}}$};
						\fill (3,3) circle (3pt) node[below] {$v_{3{(2^k-3)}}$}; 
						\fill (3,1) circle (3pt) node[below] {$v_{4{(2^k-3)}}$};
						\fill (3,0.5) circle (1pt); 
						\fill (3,0) circle (1pt); 
						\fill (3,-0.5) circle (1pt);
						\fill (3,-1) circle (3pt) node[below] {$v_{(r-1){(2^k-3)}}$};
						\fill (3,-3) circle (3pt) node[below] {$v_{r{(2^k-3)}}$};
						\fill (3,-5) circle (3pt) node[below right] {$v_{(r+1){(2^k-3)}}$};
						\fill (3,-7) circle (3pt) node[below] {$v_{(r+2){(2^k-3)}}$};
						\fill (3,-9) circle (3pt) node[below right] {$v_{(r+3)(2^k-3)}$};
						\fill (3,-11) circle (3pt) node[below] {$v_{(r+4)(2^k-3)}$};
						\fill (3,-11.5) circle (1pt);
						\fill (3,-12) circle (1pt);
						\fill (3,-12.5) circle (1pt);
						\fill (3,-13) circle (3pt) node[below right] {$v_{(\phi(n)-1)(2^k-3)}$};
						\fill (3,-15) circle (3pt) node[below] {$v_{\phi(n)(2^k-3)}$};
						
						
						\draw (6,-16) node[below]{$V_{(2^k-2)}$};
						\fill (6,7) circle (3pt) node[below] {$v_{1{(2^k-2)}}$}; 
						\fill (6,5) circle (3pt) node[below] {$v_{2{(2^k-2)}}$};
						\fill (6,3) circle (3pt) node[below] {$v_{3(2^k-2)}$}; 
						\fill (6,1) circle (3pt) node[below] {$v_{4(2^k-2)}$};
						\fill (6,0.5) circle (1pt); 
						\fill (6,0) circle (1pt); 
						\fill (6,-0.5) circle (1pt);
						\fill (6,-1) circle (3pt) node[below] {$v_{(r-1)(2^k-2)}$};
						\fill (6,-3) circle (3pt) node[below] {$v_{r(2^k-2)}$};
						\fill (6,-5) circle (3pt) node[below right] {$v_{(r+1)(2^k-2)}$};
						\fill (6,-7) circle (3pt) node[below] {$v_{(r+2)(2^k-2)}$};
						\fill (6,-9) circle (3pt) node[below right] {$v_{(r+3)(2^k-2)}$};
						\fill (6,-11) circle (3pt) node[below] {$v_{(r+4)(2^k-2)}$};
						\fill (6,-11.5) circle (1pt);
						\fill (6,-12) circle (1pt);
						\fill (6,-12.5) circle (1pt);
						\fill (6,-13) circle (3pt) node[below right] {$v_{(\phi(n)-1)(2^k-2)}$};
						\fill (6,-15) circle (3pt) node[below] {$v_{\phi(n)(2^k-2)}$};
						
						\draw (9,-16) node[below]{$V_{(2^k-1)}$};
						\fill (9,7) circle (3pt) node[below] {$v_{1{(2^k-1)}}$}; 
						\fill (9,5) circle (3pt) node[below] {$v_{2{(2^k-1)}}$};
						\fill (9,3) circle (3pt) node[below] {$v_{3(2^k-1)}$}; 
						\fill (9,1) circle (3pt) node[below] {$v_{4(2^k-1)}$};
						\fill (9,0.5) circle (1pt); 
						\fill (9,0) circle (1pt); 
						\fill (9,-0.5) circle (1pt);
						\fill (9,-1) circle (3pt) node[below] {$v_{(r-1)(2^k-1)}$};
						\fill (9,-3) circle (3pt) node[below] {$v_{r(2^k-1)}$};
						\fill (9,-5) circle (3pt) node[below right] {$v_{(r+1)(2^k-1)}$};
						\fill (9,-7) circle (3pt) node[below] {$v_{(r+2)(2^k-1)}$};
						\fill (9,-9) circle (2pt) node[below right] {$v_{(r+3)(2^k-1)}$};
						\fill (9,-11) circle (3pt) node[below] {$v_{(r+4)(2^k-1)}$};
						
						\fill (9,-11.5) circle (1pt);
						\fill (9,-12) circle (1pt);
						\fill (9,-12.5) circle (1pt);
						\fill (9,-13) circle (3pt) node[below right] {$v_{(\phi(n)-1)(2^k-1)}$};
						\fill (9,-15) circle (3pt) node[below] {$v_{\phi(n)(2^k-1)}$};
						\draw [line width=1mm] (-8,7) to [out=18,in=170](9,7); 
						\draw [line width=1mm] (-5.5,7) to (-2.5,7); 
						\draw [line width=1mm](0,7) to (0.5,7); 
						\draw [line width=1mm] (2.5,7) to (3,7); 
						
						\draw [line width=1mm](6,7) to (9,5); 

						\draw [line width=1mm] (-8,5) to (-5.5,5); 
						\draw[line width=1mm] (-2.5,5) to (0,5); 
						
						\draw [line width=1mm](3,5) to (6,5);
						
						\draw [line width=1mm] (-8,3) to [out=18,in=170](9,3); 
						\draw [line width=1mm] (-5.5,3) to (-2.5,3); 
						\draw [line width=1mm](0,3) to (0.5,3); 
						\draw [line width=1mm] (2.5,3) to (3,3); 
						
						\draw [line width=1mm](6,3) to (9,1); 
						
						
						\draw [line width=1mm] (-8,1) to (-5.5,1); 
						\draw[line width=1mm] (-2.5,1) to (0,1); 
						
						\draw [line width=1mm](3,1) to (6,1);

						
						\draw [line width=1mm] (-8,-1) to [out=18,in=170](9,-1); 
						\draw [line width=1mm] (-5.5,-1) to (-2.5,-1); 
						\draw [line width=1mm](0,-1) to (0.5,-1); 
						\draw [line width=1mm] (2.5,-1) to (3,-1); 
						\draw [line width=1mm](6,-1) to (9,-3); 

						
						\draw [line width=1mm] (-8,-3) to (-5.5,-3); 
						\draw[line width=1mm] (-2.5,-3) to (0,-3); 
						
						\draw [line width=1mm](3,-3) to (6,-3);

						\draw [line width=1mm](-8,-5) to (9,-7);

						\draw [line width=1mm](-8,-7) to (-5.5,-5);
						
						\draw [line width=1mm](-5.5,-7) to (-2.5,-5);
						
						\draw [line width=1mm](-2.5,-7) to (0,-5);
						
						\draw [line width=1mm](0,-7) to (0.75,-6.2); \draw [line width=1mm](2.25,-5.8)to (3,-5);
						
						\draw [line width=1mm](3,-7) to (6,-5);
						\draw [line width=1mm](6,-7) to (9,-5);

						\draw [line width=1mm](-8,-9) to (9,-11);

						\draw [line width=1mm](-8,-11) to (-5.5,-9);
						
						\draw [line width=1mm](-5.5,-11) to (-2.5,-9);
						
						\draw [line width=1mm](-2.5,-11) to (0,-9);
						
						\draw [line width=1mm](0,-11) to (0.75,-10.2); \draw [line width=1mm](2.25,-9.8)to (3,-9);
						
						\draw [line width=1mm](3,-11) to (6,-9);
						\draw [line width=1mm](6,-11) to (9,-9);

						\draw [line width=1mm](-8,-13) to (9,-15);

						\draw [line width=1mm](-8,-15) to (-5.5,-13);
						
						\draw [line width=1mm](-5.5,-15) to (-2.5,-13);
						
						\draw [line width=1mm](-2.5,-15) to (0,-13);
						
						\draw [line width=1mm](0,-15) to (0.75,-14.2); \draw [line width=1mm](2.25,-13.8)to (3,-13);
						
						\draw [line width=1mm](3,-15) to (6,-13);
						\draw [line width=1mm](6,-15) to (9,-13);

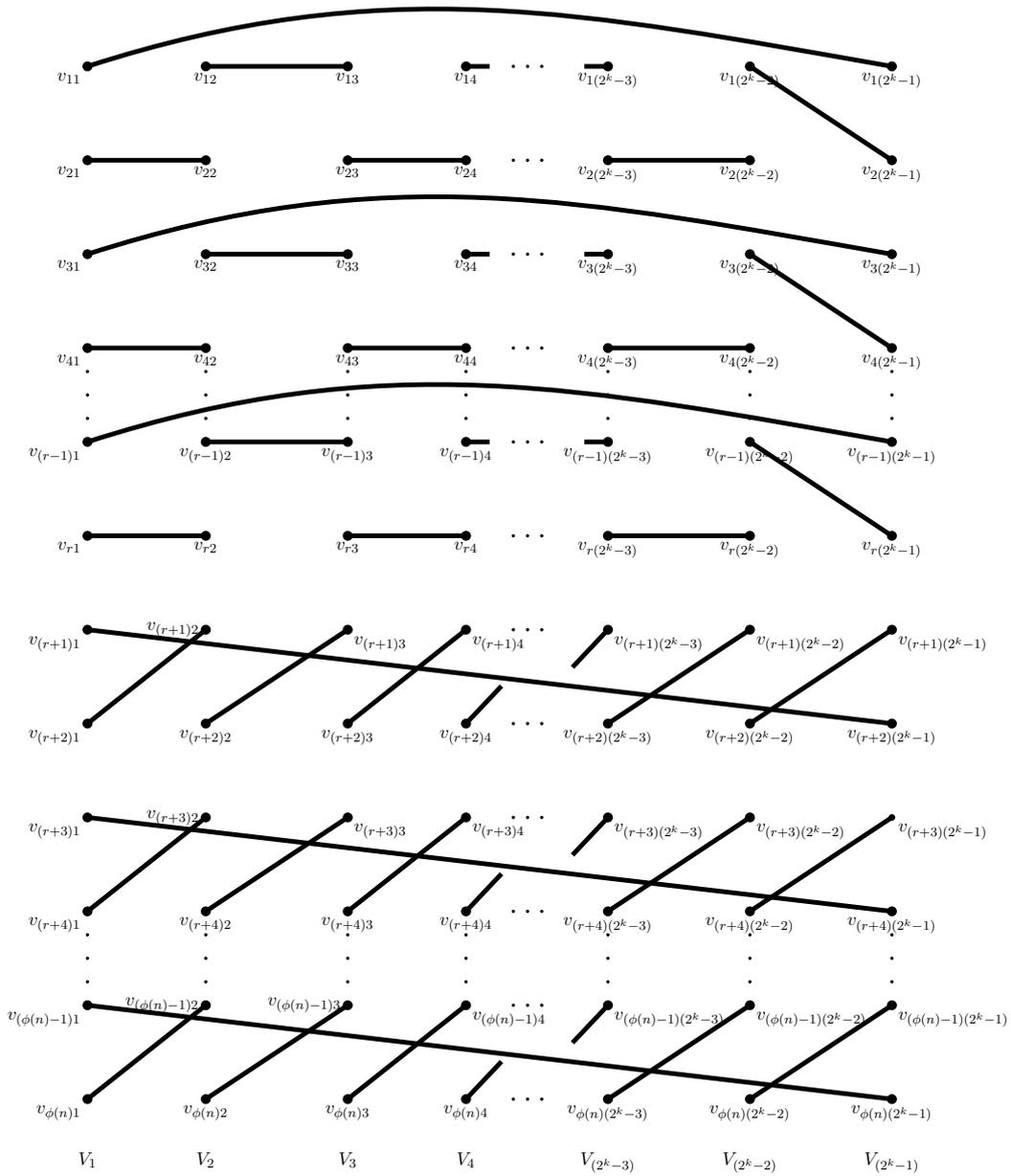
						\draw (0.5,-17) node {\textbf{Figure 5:} Perfect matching in $\mathbf{Cl_2(\mathbb{Z}_n)}$};
				\end{tikzpicture}}
			\end{center}
		\end{proof}


\begin{thebibliography}{99}
			
			\bibitem{DP} D. F. Anderson and Philip S. Livingston,\textit{ The zero divisor Graph of commutative ring}, J. Algebra \textbf{217} (1997) 434-447.
			\bibitem{IB} I. Beck, \textit{Coloring of commutative rings}, J. Algebra \textbf{116}(1) (1988) 208-226.
			\bibitem{Habibi} M. Habibi, E. Y. \c{C}elikel and C. Abdio\u{g}lu, Clean graph of a ring, J. Algebra Appl. \textbf{20}(9) (2021) 2150156.
			\bibitem{Hewitt} E. Hewitt and H. S. Zuckerman, \textit{The multiplicative semigroup of integers modulo $m$}. Pacific J. Math. \textbf{4}(1960) 1291–1308.
			\bibitem{Hosoya} H. Hosoya, \textit{Topological index. A newly proposed quantity characterizing the topological nature of structural isomers of saturated hydrocarbons}, Bull. Chem. Soc. Jpn. \textbf{4} (1971) 2332-2339.
			\bibitem{Jacobson} N. Jacobson, \textit{Basic Algebra I}, W. H. Freeman and Company, (1985).
			\bibitem{WK} W. K. Nicholson,\textit{ Lifting idempotents and exchange rings}, Trans. Amer. Math. Soc. \textbf{229} (1977) 269-278.
			\bibitem{ZZ} Z. Z. Petrovi\'{c}, and Zoran Pucanovi\'{c}, \textit{The clean graph of a commutative ring}, Ars Comb. \textbf{134} (2017) 363-378.
			\bibitem{HW} H. Wiener, \textit{Structural determination of paraffin boiling points}, J. Amer. Chem. Soc. \textbf{69} (1947) 17-20. 
		\end{thebibliography}
	\end{document}